\documentclass[12pt]{amsart}
\usepackage{amssymb, hyperref,xcolor, comment, mathrsfs, fullpage}

\newtheorem{theorem}{Theorem}[section]
\newtheorem{lemma}[theorem]{Lemma}
\newtheorem{proposition}[theorem]{Proposition}

\newtheorem{definition}[theorem]{Definition}

\theoremstyle{remark}

\numberwithin{equation}{section}

\newcommand{\mo}{{-1}}

\newcommand{\bbZ}{\ensuremath{\mathbb{Z}}}

\newcommand{\id}{\ensuremath{\mathrm{id}}}

\newcommand{\calG}{\ensuremath{\mathcal{G}}}
\newcommand{\calM}{\ensuremath{\mathcal{M}}}
\newcommand{\calN}{\ensuremath{\mathcal{N}}}
\newcommand{\calO}{\ensuremath{\mathcal{O}}}
\newcommand{\frakh}{\ensuremath{\mathfrak{h}}}
\newcommand{\scrN}{\ensuremath{\mathscr{N}}}

\begin{document}

\title{Spectrum of twists of Cayley and Cayley sum graphs}
 
\author{Arindam Biswas}
\address{Department of Mathematics, Technion - Israel Institute of Technology, Haifa 32000, Israel}
\curraddr{}
\email{biswas@campus.technion.ac.il}
\thanks{}

\author{Jyoti Prakash Saha}
\address{Department of Mathematics, Indian Institute of Science Education and Research Bhopal, Bhopal Bypass Road, Bhauri, Bhopal 462066, Madhya Pradesh,
India}
\curraddr{}
\email{jpsaha@iiserb.ac.in}
\thanks{}



\begin{abstract}
Let $G$ be a finite group with $|G|\geq 4$ and $S$ be a subset of $G$. Given an automorphism $\sigma$ of $G$, the twisted Cayley graph $C(G, S)^\sigma$ (resp. the twisted Cayley sum graph $C_\Sigma(G, S)^\sigma$) is defined as the graph having $G$ as its set of vertices and the adjacent vertices of a vertex $g\in G$ are of the form $\sigma(gs)$ (resp. $\sigma(g^{-1} s)$) for some $s\in S$. If the twisted Cayley graph $C(G, S)^\sigma$ is undirected and connected, then we prove that the nontrivial spectrum of its normalised adjacency operator is bounded away from $-1$ and this bound depends only on its degree, the order of $\sigma$ and the vertex Cheeger constant of $C(G, S)^\sigma$. Moreover, if the twisted Cayley sum graph $C_\Sigma(G, S)^\sigma$ is undirected and connected, then we prove that the nontrivial spectrum of its normalised adjacency operator is bounded away from $-1$ and this bound depends only on its degree and the vertex Cheeger constant of $C_\Sigma(G, S)^\sigma$. We also study these twisted graphs with respect to anti-automorphisms, and obtain similar results. Further, we prove an analogous result for the Schreier graphs satisfying certain conditions. 
\end{abstract}

\maketitle

% \tableofcontents

\section{Introduction}
\subsection{Motivation}
The study of the spectrum of graphs is an important theme in the theory of expanders. It was remarked by Breuillard--Green--Guralnick--Tao that the eigenvalues of the normalised Laplacian operator of non-bipartite, finite Cayley graphs are bounded away from $2$ (see \cite[Appendix E]{BGGTExpansionSimpleLie}). Recently, the first author established an explicit upper bound. 

Given a subset $S$ of a finite group $G$ with $|S| = d\geq 1$, the associated Cayley graph $C(G, S)$ has $G$ as its set of vertices and for $x, y\in G$, there is an edge from $x$ to $y$ if $y = x s$ for some $s\in S$. This graph is undirected if and only if $S$ is symmetric. In \cite[Theorem 1.4]{BiswasCheegerCayley} (cf. \cite[Theorem 2.11]{CheegerCayleySum}), it is established that if the Cayley graph $C(G, S)$ is undirected and connected, then the nontrivial spectrum of its adjacency operator lies in the interval 
$$\left( -1 + \frac{h^4}{2^9d^8} 
,
1 - \frac{h^2}{2d^2}
\right]$$
where $h$ denotes the vertex Cheeger constant of $C(G, S)$. 

It turns out that a similar result holds for the Cayley sum graphs, which are classical combinatorial objects, e.g., see \cite{HandBookCombi}. The Cayley sum graph $C_\Sigma(G, S)$ has $G$ as its set of vertices, and for $x, y\in S$, there is an edge from $x$ to $y$ if $y = x^\mo s$ for some $s\in S$. This graph is undirected if and only if $S$ is closed under conjugation (see \cite[Lemma 2.6]{CheegerCayleySum}). In \cite[Theorem 1.3]{CheegerCayleySum}, it is established that if the Cayley sum graph $C_\Sigma(G, S)$ is undirected and connected, then the nontrivial spectrum of its adjacency operator lies in the interval 
$$\left( -1 + \frac{h_\Sigma^4}{2^9d^8} 
,
1 - \frac{h_\Sigma^2}{2d^2}
\right]$$
where $h_\Sigma$ denotes the vertex Cheeger constant of $C_\Sigma(G, S)$. 

\subsection{Results obtained}
Given a group automorphism $\sigma$ of $G$, one can consider variants of the Cayley graph and the Cayley sum graph, viz., the twisted Cayley graph $C(G, S)^\sigma$ and the twisted Cayley sum graph $C_\Sigma (G, S)^\sigma$. The twisted Cayley graph $C(G, S)^\sigma$ has $G$ as its set of vertices, and for $x, y\in G$, there is an edge from $x$ to $y$ if $y = \sigma(xs)$ for some $s\in S$. The twisted Cayley sum graph $C_\Sigma (G, S)^\sigma$ has $G$ as its set of vertices, and for $x, y\in G$, there is an edge from $x$ to $y$ if $y = \sigma(x^\mo s)$ for some $s\in S$. 

Note that the twisted Cayley graphs, and the twisted Cayley sum graphs provide examples of graphs which are neither Cayley graphs, nor Cayley sum graphs (if we focus on the twists by automorphisms of order two only). In fact, there are twisted Cayley graphs which are isomorphic to no Cayley graphs, no Cayley sum graphs, no twisted Cayley sum graphs, and there are twisted Cayley sum graphs which are isomorphic to no Cayley graphs, no Cayley sum graphs, no twisted Cayley graphs, as the following examples illustrate. 

Let $p$ be an odd prime. Let $D_{2p}$ denote the dihedral group of order $2p$. Let $s$ denote an element of $D_{2p}$ of order two. Let $\sigma$ be the involution of $D_{2p}$ which fixes $s$ and sends any element of order $p$ to its inverse. Then the twisted Cayley graph $C(D_{2p}, \{s\})^\sigma$ is isomorphic to no Cayley graph, no Cayley sum graph, no twisted Cayley sum graph on any group of order $2p$. 

Let $\tau$ denote the involution of $\bbZ/2p\bbZ$ of order two. Then, for any symmetric subset $S$ of $\bbZ/2p\bbZ$ containing the identity element and having size $\leq p-1$, the twisted Cayley sum graph $C_\Sigma(\bbZ/2p\bbZ, S)^\tau$ is isomorphic to no Cayley graph, no Cayley sum graph, no twisted Cayley graph on any group of order $2p$. 

In this article, one of our aim is to show that the spectrum of the twisted Cayley graph $C(G, S)^\sigma$ and twisted Cayley sum graph $C_\Sigma(G,S)^\sigma$ are bounded away from $-1$.

\begin{theorem}\label{Thm:Bdd}
Let $S$ be a subset of a finite group $G$ with $|S|= d$. Suppose $\sigma$ is an automorphism of $G$. 
\begin{enumerate}
\item 
Suppose $\sigma^{2}$ is the trivial automorphism of $G$. 
If the twisted Cayley graph $C(G, S)^\sigma$ is connected and undirected and $|G| \geq 4$, then the nontrivial spectrum of its normalised adjacency operator lies in the interval 
$$\left( -1 + \frac{h_\sigma^4}{2^{12}d^8} 
,
1 - \frac{h_\sigma^2}{2d^2}
\right]$$
where $h_\sigma$ denotes the vertex Cheeger constant of $C(G, S)^\sigma$. 

\item 
Suppose $\sigma^{2k}$ is the trivial automorphism of $G$, where $k\geq 1$ is an odd integer. 
If the twisted Cayley graph $C(G, S)^\sigma$ is connected and undirected and $|G| \geq 4$, then the nontrivial spectrum of its normalised adjacency operator lies in the interval 
$$\left( 
\left(
-1 + \frac {1}{2^{12} d^{8k}}
\left(
\frac 12 
\left(
1 - 
\left(
1 - \frac{h_\sigma^2}{2d^2}
\right)^k
\right)
\right)
^4
\right)^{1/k}
,
1 - \frac{h_\sigma^2}{2d^2}
\right]$$
where $h_\sigma$ denotes the vertex Cheeger constant of $C(G, S)^\sigma$. 

\item If the twisted Cayley sum graph $C_\Sigma (G, S)^\sigma$ is connected and undirected and $|G| \geq 4$, then the nontrivial spectrum of its normalised adjacency operator lies in the interval 
$$\left( -1 + \frac{h_{\Sigma, \sigma}^4}{2^{12}d^8}
,
1 - \frac{h_{\Sigma, \sigma}^2}{2d^2}
\right]$$
where $h_{\Sigma, \sigma}$ denotes the vertex Cheeger constant of $C_\Sigma (G, S)^\sigma$. 
\end{enumerate}
\end{theorem}

More generally, we study the spectrum of undirected, connected, non-bipartite graphs carrying a suitable action of a group and establish Theorem \ref{thmPrincipal}. Using Theorem \ref{thmPrincipal}, we prove Theorem \ref{Thm:Bdd}. The proof of Theorem \ref{thmPrincipal} is motivated by the strategy used in \cite[Appendix E]{BGGTExpansionSimpleLie}.

In Section \ref{Sec:TwistsAnti}, we consider twisted Cayley graphs and twisted Cayley sum graphs with respect to an anti-automorphism of the underlying group. We establish that if these graphs are undirected and connected, then the nontrivial spectrum of their normalised adjacency operators are bounded away from $-1$ (see Theorem \ref{Thm:BddAnti}). 

We consider Schreier graphs in Section \ref{Sec:Schreier}. We prove under certain hypothesis that the nontrivial eigenvalues of the normalised adjacency operator of a given Schreier graph is bounded away form $-1$ (see Theorem \ref{Thm:Schreier}). 

\subsection{Acknowledgements}
We wish to thank Emmanuel Breuillard for a number of helpful discussions during the opening colloquium of the M\"unster Mathematics Cluster, 
and the MFO for their hospitality, where a part of this work was initiated. The first author is supported by the ISF Grant no. 662/15 at the Technion. The second author would like to acknowledge the Initiation Grant from the Indian Institute of Science Education and Research Bhopal and the INSPIRE Faculty Award from the Department of Science and Technology, Government of India.

\section{Preliminaries}
In the following, all the graphs considered are undirected. However, these graphs may contain multiple edges and even multiple loops at certain vertices. Given a finite $d$-regular multi-graph $\mathbb{G} = (V,E)$ having $V$ as its set of vertices and $E$ as its multiset of edges, we have the normalised adjacency operator $T$ of size $|V|\times |V|$. The normalised Laplacian operator of $\mathbb{G}$ is defined by 
$$L:= I_{|V|} - T,$$
where $I_{|V|}$ denotes the identity matrix of size $|V|\times |V|$. Let $n$ denote the number of elements of $V$. Denote the eigenvalues of $T$ and the eigenvalues of $L$ by $\lbrace t_{i} : i= 1, \cdots, n\rbrace $ and $\lbrace \lambda_{i} : i= 1, \cdots, n\rbrace $ respectively such that $\lambda_{i} = 1 - t_{i}$ and 
$$
 0 = \lambda_{1} \leqslant \lambda_{2} \leqslant \cdots \leqslant \lambda_{n-1}\leqslant \lambda_{n}  \leqslant 2.
 $$

Let $\mathbb{G} = (V,E)$ be a multi-graph. For a subset $V_{1}\subseteq V$, its neighbourhood in $\mathbb G$ is denoted by $N(V_{1})$ and is defined as
$$N(V_{1}) := \lbrace v\in V : (v, v_{1})\in E \text{ for some } v_{1}\in V_{1}\rbrace.$$
The boundary of $V_{1}$ is defined as $ \partial(V_{1}) := N(V_{1})\backslash V_{1}$.

\begin{definition}[Vertex Cheeger constant]
The vertex Cheeger constant $h(\mathbb{G})$ of a multi-graph $\mathbb{G} = (V,E)$ is defined as 
$$h(\mathbb{G}) := \inf \left\lbrace \frac{|\partial(V_{1})|}{|V_{1}|} : \emptyset \neq V_{1}\subseteq V, |V_{1}|\leqslant \frac{|V|}{2} \right\rbrace.$$
\end{definition}

We recall the notion of expander graphs as stated in \cite{AlonEigenvalueExpand}.

\begin{definition}[$(n,d,\varepsilon)$-expander]
\label{vexp}
Let $\varepsilon>0$. An $(n,d,\varepsilon)$-expander is a graph $(V,E)$ on $n$ vertices, having maximal degree $d$, such that for every set $\emptyset \neq V_{1}\subseteq V$ satisfying $|V_{1}|\leqslant \frac{|V|}{2}$, the inequality $|\partial(V_{1})|\geqslant \varepsilon|V_{1}|$ holds (equivalently, $h((V, E))\geqslant \varepsilon).$
\end{definition}

The degree of a vertex of a multi-graph is the number of half-edges adjacent to it (in the absence of loops). The presence of a loop at a vertex increases its degree by one. A multi-graph is called $r$-regular if each vertex has degree $r$. Apart from the notion of vertex expansion as in Definition \ref{vexp}, there is the notion of edge expansion.

\begin{definition}[Edge expansion]
Let $\mathbb{G} = (V,E)$ be a $d$-regular multi-graph with vertex set $V$ and edge multiset $E$. For any nonempty subset $V_1$ of $V$, the edge boundary $E(V_{1},V\backslash V_{1})$ of $V_{1}$ is defined as the multiset 
$$E(V_{1},V\backslash V_{1}) := \lbrace (v_{1},v_2)\in E: v_{1}\in V, v_2\in V\backslash V_{1} \rbrace ,$$
and the edge expansion ratio $\phi(V_{1})$ of $V_1$ is defined as 
$$\phi(V_{1}) := \frac{|E(V_{1},V\backslash V_{1})|}{d|V_{1}|}.$$
\end{definition}

\begin{definition}[Edge Cheeger constant]
The edge Cheeger constant $\mathfrak{h}(\mathbb{G})$ of a multi-graph $\mathbb G = (V, E)$ is defined by 
$$\mathfrak{h}(\mathbb{G}):= \inf_{\emptyset \neq V_{1}\subseteq V, |V_{1}|\leqslant |V|/2} \phi(V_{1}).$$
\end{definition}

The two Cheeger constants are related by the following lemma. 

\begin{lemma}
\label{Lemma:VertexEdgeCons}
For $d$-regular multi-graph $\mathbb{G} = (V,E)$, the inequalities 
$$ \frac{h(\mathbb{G})}{d} \leqslant \mathfrak{h}(\mathbb{G}) \leqslant h(\mathbb{G})$$
hold. 
\end{lemma}

\begin{proof}
For $\emptyset \neq V_{1}\subseteq V$, consider the map $$\psi:E(V_{1},V\backslash V_{1}) \rightarrow \partial(V_{1}) \text{ given by }(v_{1},v_2)\mapsto v_{2}.$$ 
This map is surjective, and hence the first inequality follows. For the second inequality, note that the fibre of any point of $\partial(V_1)$ under the map $\psi$ contains at most $d$ elements. 
\end{proof}

The following Proposition relates the edge Cheeger constant of a multi-graph with the second smallest eigenvalue of its Laplacian operator. It is the version for graphs of the corresponding inequalities for the Laplace--Beltrami operator on compact Riemannian manifolds. It was first established by Cheeger \cite{CheegerLowerBddSmallest} (the lower bound) and by Buser \cite{BuserNoteIsoperimetric} (the upper bound). The discrete version was proved by Alon and Millman \cite{AlonMilmanIsoperiIneqSupConcen} (Proposition \ref{Prop:chin}).

\begin{proposition}[Discrete Cheeger--Buser inequality]
\label{Prop:chin}
Let $\mathbb{G} = (V,E)$ be a finite $d$-regular multi-graph. 
Let $\mathfrak{h}(\mathbb{G})$ denote its edge Cheeger constant and $\lambda_{2}$ denote the second smallest eigenvalue of its normalised Laplacian operator. Then 
$$ \frac{\mathfrak{h}(\mathbb{G})^{2}}{2} \leqslant \lambda_{2} \leqslant 2\mathfrak{h}(\mathbb{G}).$$
\end{proposition}

\begin{proof}
See \cite[Propositions 4.2.4, 4.2.5]{LubotzkyDiscreteGroups} or 
\cite{ChungLaplacianOfGraph}. 
\end{proof}

\section{Generalities}

Let $(V, E)$ be a finite graph (which may contain multiple edges, and even multiple loops at certain vertices) of degree $d$. 
The neighbourhood of a subset $V'$ of $V$ in $(V, E)$ is denoted by $\calN(V')$. Assume that there exist an integer $d\geq 1$ and permutations $\theta_1, \cdots, \theta_d: V\to V$ such that the vertices $v, \theta_i(v)$ are adjacent in $(V, E)$ for any $v\in V$ and $1\leq i \leq d$, and $\calN(v)$ is equal to $\cup_{i=1}^d \{\theta_i(v)\}$ for any $v\in V$. For a subset $V'$ of $V$, denote the subset $\theta_i(V')$ of $\calN(V')$ by $\calN^i(V')$. 

\begin{proposition}
\label{Prop:AExists}
Assume that the graph $(V, E)$ is undirected. Suppose $|V|\geq 4$, the graph $(V, E)$ is an $\varepsilon$-vertex expander for some $\varepsilon>0$ and its normalised adjacency operator has an eigenvalue in the interval $(-1, -1+\zeta]$ for some $\zeta$ satisfying $0<\zeta \leq \frac{\varepsilon^2}{4d^4}$. Then for some subset $A$ of $V$, the inequalities 
$$\left(\frac 1{2 + \beta + \frac{d\beta}{\varepsilon}}\right) |V| \leq |A| \leq \frac 12 |V|$$
hold with $\beta = d^2 \sqrt{2\zeta (2-\zeta)}$. Let $\tau$ be an automorphism of the set $V$ such that 
$$
\calN (\calN(\tau(A))) \subseteq \tau(\calN(\calN(A))).
$$
If 
$$\beta < \frac{\varepsilon^2}{2d(2d+1)},$$
then exactly one of the inequalities  
$$|A \cap \tau(A)| \leq \frac{d \beta}{\varepsilon^2} (\varepsilon + d + 2)|A|, 
\quad 
|A \cap \tau(A)| \geq \left( 1 - \frac {d \beta}{ \varepsilon^2}  ( \varepsilon + d + 2) \right) |A|$$
holds.
\end{proposition}

\begin{proof}
Since the graph $(V, E)$ is an $\varepsilon$-vertex expander with $\varepsilon>0$, it follows that for some vertex $v\in V$ and for some $1\leq i \leq d$, $\theta_i(v) \neq v$. Using $|V| \geq 4$, one obtains
\begin{align*}
\varepsilon |\{v, \theta_i(v)\}|
& \leq |(\calN(v) \cup \calN(\theta_i(v)) ) \setminus \{v, \theta_i(v)\}| \\
& \leq |\calN(v) \setminus \{\theta_i(v)\}| + |\calN(\theta_i(v))\setminus \{v\}|\\
& \leq 2(d-1),
\end{align*}
which implies 
\begin{equation}
\label{Eqn:EpsilonBoundSch}
\varepsilon \leq d-1.
\end{equation}
and hence $\zeta < 1$. Let $T$ denote the normalised adjacency operator of the graph $(V, E)$. Since $T$ has an eigenvalue in $(-1, -1+\zeta]$ and $\zeta <1$, it follows that $T^2$ has an eigenvalue $\nu$ in $[(1-\zeta)^2, 1)$. 

Consider the undirected multi-graph $\calM$ (which may contain multiple edges, and multiple loops) with $V$ as its set of vertices and its edges are obtained by drawing an edge from $v$ to each element of $\calN(\calN(\{v\}))$ (considered as a multiset). In other words, $\calM$ is the undirected multi-graph having $V$ as its set of vertices and $T^2$ as its normalised adjacency operator. Thus the second largest eigenvalue of the normalised adjacency operator of $\calM$ is $\geq \nu\geq (1-\zeta)^2 = 1 - \zeta(2-\zeta)$. Hence the second smallest eigenvalue of the normalised Laplacian operator of $\calM$ is $\leq \zeta(2-\zeta)$. By the discrete Cheeger--Buser inequality (Proposition \ref{Prop:chin}), it follows that the edge Cheeger constant of $\calM$ satisfies 
$$\frac 12 \frakh(\calM)^2 \leq \zeta(2-\zeta),$$
which yields 
$$\frakh(\calM)\leq  \sqrt{2\zeta(2-\zeta)}.$$
Since $(V, E)$ has degree $d$, from Lemma \ref{Lemma:VertexEdgeCons}, it follows that the vertex Cheeger constant of $\calM$ satisfies
$$h(\calM) \leq 
d^2 \frakh(\calM) \leq d^2\sqrt{2\zeta(2-\zeta)}.$$
This implies that for some nonempty subset $A$ of $V$ with $|A| \leq \frac 12 |V|$, 
$$
|\calN(\calN(A))\setminus A| 
\leq |A| d^2 \sqrt{2\zeta(2-\zeta)} 
= |A| \beta
$$
holds. Note that $|A \cup \calN(A)| \geq \frac{|V|}{2}$. Otherwise, we obtain 
\begin{align*}
\varepsilon|A| 
& \leq \varepsilon |A \cup \calN(A)| \\
& \leq |\calN(A \cup \calN(A)) \setminus (A \cup \calN(A))|\\
& \leq |\calN(\calN(A))\setminus A| \\
& \leq |A| d^2 \sqrt{2\zeta(2-\zeta)}\\
& = |A| \beta.
\end{align*}
This implies
$\varepsilon\leq d^2 \sqrt{2\zeta(2-\zeta)} < d^2 \sqrt{4\zeta}$, which contradicts the assumption that $\zeta \leq \frac{\varepsilon^2}{4d^4}$. It follows that 
\begin{align*}
\varepsilon |V \setminus (A \cup \calN(A))|
& \leq  \varepsilon|(A \cup \calN(A))^c|\\
& \leq |\calN((A \cup \calN(A))^c) \setminus (A \cup \calN(A))^c|  \\
& \leq |(A \cup \calN(A)) \cap \calN((A \cup \calN(A))^c)| \\
& \leq \sum_{i = 1}^d 
|(A \cup \calN(A)) \cap \theta_i((A \cup \calN(A))^c)| \\
& \leq \sum_{i = 1}^d |\calN(A \cup \calN(A)) \cap (A \cup \calN(A))^c|\\
& \leq d |\calN(A \cup \calN(A)) \cap (A \cup \calN(A))^c|\\
& \leq d |\calN(A \cup \calN(A)) \setminus (A \cup \calN(A))|\\
& \leq d|\calN(\calN(A))\setminus A| \\
& \leq d|A| d^2 \sqrt{2\zeta(2-\zeta)}\\
& = d |A| \beta,
\end{align*}
which implies 
\begin{align*}
|V| 
& \leq \frac{d\beta}{\varepsilon}|A| + |A \cup \calN(A)|\\
& \leq \frac{d\beta}{\varepsilon}|A| + |A| + |\calN(A)| \\
& \leq \frac{d\beta}{\varepsilon}|A| + |A| + |\calN(\calN(A))| \\
& \leq \frac{d\beta}{\varepsilon}|A| + 2|A| + |\calN(\calN(A))\setminus A|\\
& \leq \frac{d\beta}{\varepsilon}|A| + 2|A| + \beta|A|,
\end{align*}
and hence 
$$\frac{|V|}{2 + \beta + \frac{d\beta}{\varepsilon}} \leq |A| \leq \frac{|V|}{2}.$$

Note that the inequalities 
$$
\frac {2d\beta} { \varepsilon^2} ( \varepsilon + d + 2) 
\leq \frac {2d\beta} { \varepsilon^2}( d -1 + d + 2) 
= \frac {2d\beta} { \varepsilon^2} ( 2d + 1) 
<1 
$$
imply that 
$$\frac{d \beta}{\varepsilon^2} (\varepsilon + d + 2)
< 
1 - \frac {d \beta}{ \varepsilon^2}  ( \varepsilon + d + 2).$$
Hence it suffices to show that one of the inequalities
$$|A \cap (\tau(A))| \leq \frac{d \beta}{\varepsilon^2} (\varepsilon + d + 2)|A|, 
\quad 
|A \cap (\tau(A))| \geq \left( 1 - \frac {d \beta}{ \varepsilon^2}  ( \varepsilon + d + 2) \right) |A|$$
holds.

Let $B = A\Delta(\tau(A))^c$. Let $\id$ denote the identity map from $V$ to $V$. Note that 
\begin{align*}
|\calN^i(B) \Delta B| 
& \leq |\calN^i(A) \Delta \calN^i((\tau(A))^c) \Delta A \Delta(\tau(A))^c| \\
& \leq |\calN^i(A) \Delta \calN^i((\tau(A))^c) \Delta A^c \Delta \tau(A)| \\
& \leq |\calN^i(A) \Delta A^c \Delta \calN^i((\tau(A))^c) \Delta \tau(A)| \\
& \leq |\calN^i(A) \Delta A^c| + |\calN^i((\tau(A))^c) \Delta \tau(A)|\\
& = |\calN^i(A) \Delta A^c| + |\calN^i(\tau(A)) \Delta \tau(A^c)|\\
& = \sum_{\phi\in \{\id, \tau\}} |\calN^i(\phi(A)) \Delta \phi(A^c)|\\
& = \sum_{\phi\in \{\id, \tau\}} \left( |\calN^i(\phi(A))| + |\phi(A^c)| - 2|\calN^i(\phi(A)) \cap \phi(A^c)|\right)\\
& = \sum_{\phi\in \{\id, \tau\}} \left( |V| - 2|\calN^i(\phi(A)) \cap \phi(A^c)|\right)\\
& = \sum_{\phi\in \{\id, \tau\}} \left( |V| - 2|\calN^i(\phi(A))| + 2|\calN^i(\phi(A))|  - 2|\calN^i(\phi(A)) \cap \phi(A^c)|\right)\\
& = \sum_{\phi\in \{\id, \tau\}} \left( |V| - 2|A| + 2|\calN^i(\phi(A)) \cap \phi(A)|\right)\\
& = 2(|V| - 2|A|) + 2\sum_{\phi\in \{\id, \tau\}} |\calN^i(\phi(A)) \cap \phi(A)| \\
& \leq 2(|V| - 2|A|) + 
\frac 2\varepsilon\sum_{\phi\in \{\id, \tau\}} \varepsilon|\calN(\phi(A)) \cap \phi(A)| \\
& \leq 2(|V| - 2|A|) + 
\frac 2\varepsilon\sum_{\phi\in \{\id, \tau\}} |\calN(\calN(\phi(A)) \cap \phi(A)) \setminus (\calN(\phi(A)) \cap \phi(A))| \\
& \leq 2(|V| - 2|A|) + 
\frac 2\varepsilon\sum_{\phi\in \{\id, \tau\}} |\calN(\calN(\phi(A)) \setminus \phi(A)| \\
& \leq 2(|V| - 2|A|) + 
\frac 2\varepsilon\sum_{\phi\in \{\id, \tau\}} |\phi(\calN(\calN(A))) \setminus \phi(A)| \\
& = 2(|V| - 2|A|) + 
\frac 2\varepsilon\sum_{\phi\in \{\id, \tau\}} |\calN(\calN(A)) \setminus A| \\
& \leq 2\left(\beta + \frac{d\beta}\varepsilon \right)|A| +\frac 2\varepsilon 2|A|\beta\\
& =  \frac{2\beta}\varepsilon (\varepsilon + d + 2)|A| .
\end{align*}
It follows that 
$$
|\calN(B) \Delta B|
\leq 
\frac{2d\beta}\varepsilon (\varepsilon + d + 2)|A| .
$$

Note that 
\begin{align*}
|\calN^i(B^c) \Delta B^c| 
& \leq |\calN^i(A) \Delta \calN^i(\tau(A)) \Delta A \Delta\tau(A)| \\
& \leq |\calN^i(A) \Delta \calN^i(\tau(A)) \Delta A^c \Delta (\tau(A))^c| \\
& \leq |\calN^i(A) \Delta A^c \Delta \calN^i(\tau(A)) \Delta (\tau(A))^c| \\
& \leq |\calN^i(A) \Delta A^c| + |\calN^i(\tau(A)) \Delta (\tau(A))^c| )\\
& \leq |\calN^i(A) \Delta A^c| + |\calN^i((\tau(A)^c)) \Delta \tau(A)| )\\
& \leq \frac{2\beta}\varepsilon (\varepsilon + d + 2)|A| .
\end{align*}
It follows that 
$$
|\calN(B^c) \Delta B^c|
\leq 
\frac{2d\beta}\varepsilon (\varepsilon + d + 2)|A| .
$$

We consider the following cases, viz., $|B|\leq \frac{|V|}2, |B| > \frac{|V|}2$. When $|B|\leq \frac{|V|}2$ holds, we obtain 
$$
\varepsilon |B| \leq |\calN(B) \setminus B| \leq |\calN(B) \Delta B|  \leq \frac{2d \beta}\varepsilon (\varepsilon + d + 2)|A|,
$$
which yields 
$$|B| \leq \frac {2d\beta}{ \varepsilon^2} ( \varepsilon + d + 2) |A|.
$$
Since 
\begin{align*}
|V| - |B| 
& = |B^c| \\
& = |A \Delta (\tau (A))| \\
& = |A| - |A\cap (\tau(A))| + |\tau(A)| - |A\cap (\tau(A))| \\
& = 2|A| - 2|A\cap (\tau(A))| 
\end{align*}
holds, we obtain 
$$
2|A \cap (\tau(A)) |
\leq |V| - 2|A| + 2|A \cap (\tau(A))|
= |B| 
\leq \frac{2d \beta}{\varepsilon^2} (\varepsilon + d + 2)|A|.
$$
While $|B| > \frac{|V|}2$ holds, we obtain 
$$
\varepsilon |B^c| \leq |\calN(B^c)\setminus B^c| \leq |\calN(B^c)\Delta B^c|  
\leq \frac{2d\beta}{\varepsilon}(\varepsilon + d + 2) |A|,
$$
which yields 
$$|B^c| \leq \frac{2d\beta}{\varepsilon^2}(\varepsilon + d + 2) |A|.
$$
Since 
$$
|B^c| 
= |A \Delta (\tau(A))| 
= |A| - |A\cap (\tau(A))|  + |(\tau(A)) | - |A\cap (\tau(A))| 
= 2|A| - 2|A\cap (\tau(A))| $$
holds, we obtain 
$$
|A \cap (\tau(A)) |
\geq |A| - \frac {d\beta}{ \varepsilon^2} ( \varepsilon + d + 2)|A|
= \left( 1- \frac {d\beta}{ \varepsilon^2} ( \varepsilon + d + 2)\right)|A|.
$$
This completes the proof. 
\end{proof}

\begin{theorem}\label{thmPrincipalk1}
Suppose $V$ carries a left action of a group $\calG$ such that the following conditions hold. 
\begin{enumerate}
\item No index two subgroup of $\calG$ acts transitively on $V$.

\item The action of $\calG$ on the set $V$ is ``transitive of order $t$'' in the sense that for each $(u, v)\in V\times V$, the equation $gu = v$ has exactly $t$ distinct solutions for $g\in \calG$. In other words, the action of $\calG$ on $V$ is transitive and the stabilizer of each element of $V$ have the same size (equal to $t$).

\item For each $\theta_i, 1\leq i\leq d$ and $v\in V$, there is an automorphism or an anti-automorphism $\psi_{i,v}$ of the group $\calG$ such that one of 
$$\theta_i(g\cdot v) =  \psi_{i,v}(g) \cdot \theta_i(v)$$
and 
$$\theta_i(g\cdot v) =  \psi_{i, v}(g^\mo) \cdot \theta_i(v)$$
holds for any $g\in \calG$. 

\item For any $\tau\in \calG$, 
$
\calN (\calN(\tau(A))) \subseteq \tau(\calN(\calN(A)))
$
holds. 
\end{enumerate}
Assume that the graph $(V, E)$ is undirected and non-bipartite. Assume further that $|V|\geq 4$, the graph $(V, E)$ is an $\varepsilon$-vertex expander for some $\varepsilon>0$. Then the nontrivial eigenvalues of the normalised adjacency operator of this graph are greater than $-1 + \ell_{\varepsilon, d}$ with 
$$\ell_{\varepsilon, d}
=\frac{\varepsilon^4}{2^{12} d^8}.$$
\end{theorem}

\begin{proof}
On the contrary, let us assume that a nontrivial eigenvalue of the normalised adjacency operator of the graph $(V, E)$ lies in the interval $\left[-1, -1 +  \ell_{\varepsilon, d}\right]$. Since $(V, E)$ is non-bipartite, it follows that $-1$ is not an eigenvalue of its normalised adjacency operator. Hence an eigenvalue of the normalised adjacency operator of the graph $(V, E)$ lies in the interval $(-1, -1+ \ell_{\varepsilon, d}]$. Set 
\begin{align*}
\gamma & = d^2 \sqrt{2 \ell_{\varepsilon, d}(2- \ell_{\varepsilon, d})},\\
r & =  1- \frac {d \gamma}{ \varepsilon^2}  ( \varepsilon + d + 2).
\end{align*}
Since $\ell_{\varepsilon, d} = \frac{\varepsilon^4}{2^{12} d^8}$, we have 
$$\gamma =  d^2 \sqrt{2 \ell_{\varepsilon, d}(2- \ell_{\varepsilon, d})} <  d^2 \sqrt{4 \ell_{\varepsilon, d}} \leqslant \frac{\varepsilon^{2}}{2^5d^2}.$$
$$
1 - r = \frac {d \gamma}{ \varepsilon^2}  ( \varepsilon + d + 2)
\leq \frac {d \gamma}{ \varepsilon^2}  (2d+1) 
< \frac 3{2^3\sqrt 2} < \frac 13.$$
Consequently, 
\begin{equation}
\label{Eqn:BoundssigmaCayleySum}
\ell_{\varepsilon, d}\leq \frac{\varepsilon^2}{4d^4}, 
\gamma < \frac{\varepsilon^2}{2d(2d+1)} \text{ and }  r> \frac 23.
\end{equation}
 
Define the subset $H$ of $\calG$ by 
$$H  :=
\{
\tau\in \calG \,:\,\, |A \cap  (\tau(A))| \geq r|A|
\}.$$
Note that $H$ contains the identity element of $\calG$. By the triangle inequality, 
\begin{align*}
|A \setminus (\tau_1(\tau_2(A)))| 
& \leq |A \setminus (\tau_1(A)) | + |(\tau_1(A)) \setminus (\tau_1(\tau_2(A)))| \\
& =  |A \setminus  (\tau_1(A)) | + |A  \setminus (\tau_2(A)) |\\
& = |A | -  |A \cap (\tau_1(A))  | + |A | - |A  \cap (\tau_2(A)) |\\
& \leq 2|A| - 2r |A| .
\end{align*}
Consequently, 
\begin{align*}
|A \cap (\tau_1(\tau_2(A)))| 
& = |A | - |A \setminus (\tau_1(\tau_2(A)))| \\
& \geq |A | - 2|A| + 2r |A| \\
& = (2r - 1) |A|.
\end{align*}
If $|A \cap (\tau_1(\tau_2(A)))| \leq (1-r) |A|$, then we obtain 
$$(1-r) |A|
\geq |A \cap (\tau_1(\tau_2(A)))| 
\geq  (2r - 1) |A|,$$
which implies $r\leq \frac 23$. Since $r>\frac 23$, it follows that the inequality 
$|A \cap (\tau_1(\tau_2(A)))| \leq (1-r) |A|$
does not hold. Hence, by Proposition \ref{Prop:AExists}(1), $H$ contains $\tau_1\tau_2$. So $H$ is a subgroup of $\calG$. 

For any $g\in \calG$, the map
$$A\cap g^\mo A \to A\times A, \quad 
x\mapsto (x, gx)$$ 
is well-defined, and induces a map 
$$\varphi: \coprod _{g\in \calG} A\cap g^\mo A \to A\times A.$$
Each of its fibre contains exactly $t$ elements and hence 
$$
|A|^2
= im(\varphi)
= \sum_{x\in im(\varphi) } \frac{|\varphi^\mo (x)|} {|\varphi^\mo (x)|}
= \frac 1t \sum_{x\in im(\varphi)} |\varphi^\mo (x)|
= \frac 1t \sum_{g\in \calG} |A\cap g^\mo A|.$$

If $\calG = H$, we obtain 
$$|A| \cdot \frac{|V|}{2} 
\geq |A|^2 = \frac 1t \sum_{g\in \calG} |A\cap g^\mo A|
\geq \frac 1t |\calG| \cdot r|A|,$$
which gives 
$$|\calG| \leq \frac t{2r}{|V|}.$$
This contradicts that $|\calG| \geq t|V|$. Hence $H$ is a proper subgroup of $\calG$. 

The following estimate 
$$
t |A|^2 
= \sum_{g\in \calG} |A\cap g^\mo A| 
\leq |H| |A| + \frac{d\gamma}{\varepsilon^2}(\varepsilon + d+ 2)|A||\calG\setminus H| $$
implies 
$$t |A|  
\leq |H| + \frac{d\gamma}{\varepsilon^2}(\varepsilon + d+ 2)(|\calG| - |H|).$$
Using Proposition \ref{Prop:AExists}(1), we obtain 
$$
\left(\frac{1}{2+ \gamma + \frac{d\gamma}{\varepsilon}}\right) |\calG| 
= \left(\frac{t}{2+ \gamma + \frac{d\gamma}{\varepsilon}}\right) |V| 
\leq \frac{d\gamma}{\varepsilon^2}(\varepsilon + d+ 2)|\calG| 
+  \left(1 - \frac{d\gamma}{\varepsilon^2}(\varepsilon + d+ 2)\right) |H|.$$
We claim that $H$ is a subgroup of $\calG$ of index two. To prove this claim, it suffices to show that 
\begin{equation}
\label{Eqn:13rdInePri}
\frac 13 \left(1 - \frac{d\gamma}{\varepsilon^2}(\varepsilon + d+ 2)\right)
< 
\left(\frac{1}{2+ \gamma + \frac{d\gamma}{\varepsilon}}\right) - \frac{d\gamma}{\varepsilon^2}(\varepsilon + d+ 2),
\end{equation}
i.e., 
$$\left(2 + \gamma + \frac{d\gamma}{\varepsilon}\right) \left( 1 + \frac{2d\gamma}{\varepsilon^2}(\varepsilon + d+ 2) \right) < 3,$$
which is equivalent to 
\begin{equation}
\label{Eqn:13rdInequality}
\left(\gamma + \frac{d\gamma}{\varepsilon}\right) + \frac{2d\gamma}{\varepsilon^{2}}(\varepsilon + d+ 2)\left(2 + \gamma + \frac{d\gamma}{\varepsilon}\right) < 1.
\end{equation}

Assume that $\zeta \leq \frac{\varepsilon^4}{2^9 d^8}$. Since 
$$\gamma \leq
\frac{\varepsilon^{2}}{2^3\sqrt{2}d^2}
$$
and $d \geq 2$ (by Equation \eqref{Eqn:EpsilonBoundSch}), we obtain 
\begin{align*}
& \gamma + \frac{d\gamma}{\varepsilon} + \frac{2d\gamma}{\varepsilon^{2}}(\varepsilon + d+ 2)\left(2 + \gamma + \frac{d\gamma}{\varepsilon}\right)\\
& = \frac{\gamma}{\varepsilon}(\varepsilon + d) + \frac{2d\gamma}{\varepsilon^{2}}(\varepsilon + d+ 2)\left(2 + \frac{\gamma}{\varepsilon}(\varepsilon+ d)\right) \\
& = \frac{\gamma}{\varepsilon}\left(\varepsilon +d+ \frac{4d(\varepsilon +d+ 2)}{\varepsilon} \right) + \frac{2d\gamma^2}{\varepsilon^3}(\varepsilon+d)(\varepsilon +  d + 2)\\
& = \frac{\gamma}{\varepsilon^2}\left(\varepsilon(\varepsilon +d)+ 4d(\varepsilon +d+ 2)\right) + \frac{2d\varepsilon \gamma^2}{\varepsilon^4}(\varepsilon+d)(\varepsilon +  d + 2)\\
& \leq \frac{\gamma}{\varepsilon^2}\left((d-1) (2d-1) + 4d(2d+1) \right) + \frac{2d\gamma^2}{\varepsilon^4} (d-1)(2d-1)(2d+1)\\
& = \frac{1}{8\sqrt 2 d^2}(10d^2 + d +1) + \frac{1}{64d^3}(d-1)(2d-1)(2d+1)\\
& = \frac{(d-1)(4d^2-1) + 4\sqrt 2 d (10d^2 + d + 1)}{64d^3}\\
& \leq \frac{(d-1)(4d^2-1) + 41\sqrt 2 d^3 + 8\sqrt 2 d}{64d^3}\\
& \leq \frac{(41\sqrt 2 + 4) d^3 + 8\sqrt 2 d - 4d^2 - d + 1}{64d^3}\\
& \leq \frac{(41\sqrt 2 + 4) d^3 + 4d(2 \sqrt 2  - d) + (1-d)}{64d^3}\\
& < \frac{41\sqrt 2 + 4}{64} + \frac{\sqrt 2 - 1} {32} \\
& < 1.
\end{align*}
So the claim that $H$ is a subgroup of $\calG$ of index two follows. Since no index two subgroup of $\calG$ acts transitively on $V$, the action of $H$ on $V$ has at least two distinct orbits. Since the action of $\calG$ on $V$ is transitive, the action of $H$ on $V$ has exactly two orbits. Let $\calO_1, \calO_2$ denote the orbits of the action of $H$ on $V$. 

Note that for any $g\in H$ and for any subset $B$ of $V$ contained in $\calO_1$ or $\calO_2$, the map 
$$B\cap g^\mo B \to B\times B, \quad 
x\mapsto (x, gx)$$ 
is well-defined, and induces a map 
$$\varphi: \coprod _{g\in H} B\cap g^\mo B \to B\times B.$$
Each of its fibre contains exactly $t$ elements and hence 
$$
|B|^2
= im(\varphi)
= \sum_{x\in im(\varphi) } \frac{|\varphi^\mo (x)|} {|\varphi^\mo (x)|}
= \frac 1t \sum_{x\in im(\varphi)} |\varphi^\mo (x)|
= \frac 1t \sum_{g\in H} |B\cap g^\mo B|.$$
For any $g\in H$, we have 
\begin{align*}
A\cap g^\mo A
& = 
((A\cap \calO_1) \cup (A\cap \calO_2))
\cap 
(g^\mo ((A\cap \calO_1) \cup (A\cap \calO_2)))\\
& = 
((A\cap \calO_1) \cap 
(g^\mo (A\cap \calO_1))
\cup 
((A\cap \calO_2) \cap 
(g^\mo (A\cap \calO_2)),
\end{align*}
which yields 
$$
|A \cap g^\mo A|  = \sum_{B = A\cap \calO_1, A\cap \calO_2}|B\cap g^\mo B| .$$
So 
\begin{align*}
|A\cap \calO_1|^2 + |A\cap \calO_2|^2  
& = \sum_{B = A\cap \calO_1, A\cap \calO_2} |B|^2\\
& = \frac 1t \sum_{B = A\cap \calO_1, A\cap \calO_2}\sum_{g\in H} |B\cap g^\mo B| \\
& = \frac 1t \sum_{g\in H} \sum_{B = A\cap \calO_1, A\cap \calO_2}|B\cap g^\mo B| \\
& = \frac 1t \sum_{g\in H} |A \cap g^\mo A| .
\end{align*}
It follows that 
\begin{align*}
|A|^2 - 2|A\cap \calO_1||A\cap \calO_2| 
& = (|A\cap \calO_1| + |A\cap \calO_2|)^2 - 2|A\cap \calO_1||A\cap \calO_2| \\
& = |A\cap \calO_1|^2 + |A\cap \calO_2|^2  \\
& = \frac 1t \sum_{g\in H} |A\cap gA| \\
& = \frac 1t \sum_{g\in \calG} |A\cap gA| - \frac 1t \sum_{g\in H^c} |A\cap gA| \\ 
& = |A|^2 - \frac 1t \sum_{g\in H^c} |A\cap gA| \\ 
& \geq |A|^2 - \sum_{g\in H^c} \frac{d \gamma}{t\varepsilon^2} (\varepsilon + d + 2)|A| \\
& = |A|^2 - \frac{d \gamma}{t\varepsilon^2} (\varepsilon + d + 2)|A||H|.
\end{align*}
This implies that 
$$
|A\cap \calO_1||A\cap \calO_2| 
\leq \frac{d \gamma}{4t\varepsilon^2} (\varepsilon + d + 2)|A||\calG|
\leq \frac{d \gamma}{4t\varepsilon^2} (\varepsilon + d + 2) \frac{t|V|^2}{2}\leq \frac{d \gamma}{8\varepsilon^2} (\varepsilon + d + 2)|V|^2.
$$
Hence, for the orbit $\calO$ of some element of $V$ under the action of $H$, we have 
$$
|A\cap \calO^c| \leq 
\sqrt{\frac{d\gamma}{8\varepsilon^2} (\varepsilon + d + 2) }|V|.$$
For any $1\leq i \leq d$, 
\begin{align*}
|\theta_i(\calO)\cap \calO |
& = |\theta_i(\calO\cap A) \cap \calO| + |\theta_i(\calO\setminus A) \cap \calO|\\
& \leq |\theta_i(A) \cap \calO| + |\calO\setminus A|\\
& = |\theta_i(A) \cap (\calO\cap A) | + |\theta_i(\calO\cap A) \cap (\calO\setminus A) | + |\calO\setminus A|\\ 
& \leq |\theta_i(A) \cap A| + 2|\calO\setminus A|\\ 
& = |\theta_i(A) \cap A| + 2|\calO| - 2|\calO\cap A|\\ 
& = |\theta_i(A) \cap A| + |V| - 2|A| + 2|A\cap \calO^c|\\ 
& \leq \frac \gamma\varepsilon |A| + \left(2 + \gamma + \frac{d\gamma}{\varepsilon}\right)|A| - 2|A| + 2\sqrt{\frac{d\gamma}{8\varepsilon^2} (\varepsilon + d + 2) }|V| \\
& = \frac \gamma\varepsilon |A| + \left(\gamma + \frac{d\gamma}{\varepsilon}\right)|A|  + 2\sqrt{\frac{d\gamma}{8\varepsilon^2} (\varepsilon + d + 2) }|V| \\
& \leq \left(\frac \gamma\varepsilon (d+1 + \varepsilon)  + \sqrt{\frac{2d\gamma}{\varepsilon^2} (\varepsilon + d + 2) }\right) \frac{|V|}{2} \\
& \leq \left(\frac {2d\gamma}\varepsilon |H| + \sqrt{\frac{2d\gamma}{\varepsilon^2} (2d + 1) }\right) \frac{|V|}{2} \\
& \leq \left(\frac {2d\gamma}\varepsilon  + \sqrt{\frac{6d^2\gamma}{\varepsilon^2}} \right) \frac{|V|}{2}\\
& < \left(\frac {1}{2^4}  + \sqrt{\frac{6}{2^5}} \right) \frac{|V|}{2}\\
& < \frac{|V|}{4} \\
& = \frac{|\calG|}{4t} \\
& = \frac{|H|}{2t}.
\end{align*}

Note that for any two subgroup $H_1, H_2$ of $\calG$ of index two, the inequalities 
$$|H_1\cap H_2| \geq \frac{|H_1|}{2}, |H_1^c\cap H_2^c| \geq \frac{|H_1|}{2}$$
hold. Indeed, for any $x\in H_1 \cap H_2^c$, the `left multiplication by $x$' map induces an injection from $H_1\cap H_2^c\to H_1 \cap H_2$. This implies that  
$$ |H_1\cap H_2| 
\geq |H_1 \cap H_2^c|,$$
which yields
$$ |H_1\cap H_2| 
\geq \frac 12 ( |H_1\cap H_2^c|  + |H_1\cap H_2| ) = \frac{|H_1|}{2}.$$
It follows that $|H_1\cap H_2^c| \leq \frac{|H_1|}2$, which implies $|H_1^c\cap H_2^c| \geq \frac{|H_1|}2$. 

Note that no two elements of $\calO$ are adjacent in $(V, E)$. Otherwise, for $u, v\in \calO$, $u = \theta_i(v)$ for some $i$ and hence 
\begin{align*}
|\theta_i(\calO)\cap \calO|
& = |\theta_i(Hv) \cap Hu | \\
& = | \psi_{i,v} (H) \theta_i(v) \cap H u| \\
& = |\psi_{i,v} (H) u \cap H u|\\
& \geq |(\psi_{i,v} (H) \cap H) u|\\
& \geq \frac{|\psi_{i,v} (H) \cap H|}{t} \\
& \geq \frac{|H|}{2t}.
\end{align*}
Moreover, no two elements of $\calO^c$ are adjacent in $(V, E)$. Otherwise, for $u, v\in \calO^c$, $u = \theta_i(v)$ for some $i$ and hence
\begin{align*}
|\theta_i(\calO)\cap \calO|
& = |\theta_i(H^c v) \cap H^c u | \\
& = |\psi_{i,v} (H^c) \theta_i(v) \cap H^c u|\\
& = |\psi_{i,v} (H^c) u \cap H^c u| \\
& \geq |(\psi_{i,v} (H^c) \cap H^c) u| \\
& \geq \frac{|\psi_{i,v} (H^c) \cap H^c|}{t} \\
& \geq \frac{|H|}{2t}.
\end{align*}
So $(V, E)$ is bipartite, contradicting the hypothesis. Hence, the nontrivial eigenvalues of the normalised adjacency operator of this graph are greater than $-1 + \ell_{\varepsilon, d}$ with 
$$\ell_{\varepsilon, d}
=\frac{\varepsilon^4}{2^{12} d^8}.$$
\end{proof}

In the following, $(V, E)^k$ denotes the graph having $V$ as its set of vertices and the $k$-th power of the adjacency operator of $(V,E)$ as its adjacency operator. 

\begin{theorem}
\label{thmPrincipal}
Suppose $V$ carries a left action of a group $\calG$, $(V, E)$ is undirected, $|V|\geq 4$ and $k\geq 1$ be an odd integer such that the following conditions hold. 
\begin{enumerate}
\item No index two subgroup of $\calG$ acts transitively on $V$.

\item The action of $\calG$ on the set $V$ is ``transitive of order $t$'' in the sense that for each $(u, v)\in V\times V$, the equation $gu = v$ has exactly $t$ distinct solutions for $g\in \calG$. In other words, the action of $\calG$ on $V$ is transitive and the stabilizer of each element of $V$ have the same size (equal to $t$).

\item 
$(V, E)^k$ is an $\varepsilon_k$-vertex expander with $\varepsilon_k >0$. 

\item 
For $1\leq i_1, \cdots, i_k \leq d$ and $v\in V$, there is an automorphism or an anti-automorphism $\psi_{(i_1, \cdots, i_k), v}$ of the group $\calG$ such that one of 
$$
(\theta_{i_1} \circ \cdots \circ \theta_{i_k} )(g\cdot v) 
= 
\psi_{(i_1, \cdots, i_k),  v}(g) \cdot 
(\theta_{i_1} \circ \cdots  \circ \theta_{i_k} )(v) 
$$
and 
$$
(\theta_{i_1} \circ \cdots  \circ \theta_{i_k} )(g\cdot v) 
= 
\psi_{(i_1, \cdots, i_k),  v}(g^\mo) \cdot 
(\theta_{i_1} \circ \cdots \circ \theta_{i_k} )(v) 
$$
holds for any $g\in \calG$. 

\item For any $\tau\in \calG$ and for any subset $A$ of $V$, 
$
\calN^k (\calN^k(\tau(A))) \subseteq \tau(\calN^k(\calN^k(A)))
$
holds. 
\end{enumerate}
Then the nontrivial eigenvalues of the normalised adjacency operator of $(V, E)$ are greater than 
$$
\left(
-1 + \frac {1}{2^{12} d^{8k}}
\varepsilon_k^4
\right)^{1/k}.
$$
If $(V, E)$ is an $\varepsilon$-vertex expander with $\varepsilon >0$, then the vertex Cheeger constant of $(V, E)^k$ is
$$
\geq 
\frac 12 
\left(
1 - 
\left(
1 - \frac{\varepsilon^2}{2d^2}
\right)^k
\right)
.$$
\end{theorem}

\begin{proof}
By Theorem \ref{thmPrincipalk1}, the nontrivial spectrum of the adjacency operator of $(V, E)^k$ is bounded 
away from $-1 + \frac {\varepsilon_k^4}{2^{12} d^{8k}}$. Since $k$ is odd, the nontrivial spectrum of $(V, E)$ is bounded away from 
$$
\left(
-1 + \frac {1}{2^{12} d^{8k}}
\varepsilon_k^4
\right)^{1/k}
.$$

Since $(V, E)$ is an $\varepsilon$-vertex expander, the largest nontrivial eigenvalue of its adjacency operator is bounded by $1 - \frac{\varepsilon^2}{2d^2}$. Since $k$ is odd, the largest nontrivial eigenvalue of the adjacency operator of $(V, E)^k$ is bounded by  
$$
\left(
1 - \frac{\varepsilon^2}{2d^2}
\right)^k
.$$
By Lemma \ref{Lemma:VertexEdgeCons} and the discrete Cheeger--Buser inequality (Proposition \ref{Prop:chin}), the vertex Cheeger constant of $(V, E)^k$ is
$$
\geq 
\frac 12 
\left(
1 - 
\left(
1 - \frac{\varepsilon^2}{2d^2}
\right)^k
\right)
.$$
\end{proof}

\section{Spectral expansion of the Cayley and Cayley sum graphs twisted by automorphisms}
\label{Sec:Twists}

Let $G$ be a finite group, $S$ be a subset of $G$ and $\sigma$ be a group automorphism of $G$.

Consider the twist $C(G, S)^\sigma$ of the Cayley graph $C(G, S)$ by the automorphism $\sigma$. The graph $C(G, S)^\sigma$ has $G$ as its set of vertices, and there is an edge from $x$ to $y$ whenever $y = \sigma(xs)$ for some $s\in S$. Roughly speaking, the twisted Cayley graph $C(G, S)^\sigma$ has the same set of vertices as that of the Cayley graph $C(G, S)$, and given a vertex $x$ in $C(G, S)^\sigma$, its adjacent vertices are precisely the translates of the adjacent vertices of $x$ in $C(G, S)$ under $\sigma$. 

Consider the twist $C_\Sigma(G, S)^\sigma$ of the Cayley sum graph $C_\Sigma(G, S)$ by the automorphism $\sigma$. The graph $C_\Sigma(G, S)^\sigma$ has $G$ as its set of vertices, and there is an edge from $x$ to $y$ whenever $y = \sigma(x^\mo s)$ for some $s\in S$. Roughly speaking, the twisted Cayley sum graph $C_\Sigma(G, S)^\sigma$ has the same set of vertices as that of the Cayley sum graph $C_\Sigma(G, S)$, and given a vertex $x$ in $C_\Sigma(G, S)^\sigma$, its adjacent vertices are precisely the translates of the adjacent vertices of $x$ in $C_\Sigma(G, S)$ under $\sigma$. 

\subsection{The twisted Cayley graph}

For a subset $A$ of $G$, its neighbourhood in $C(G, S)^\sigma$ is denote by $\scrN(A)$. Given $g_1, \cdots, g_r \in G$,  
$\prod_{i = 1}^r g_i$
denotes the product 
$g_1 \cdots g_r$. 

\begin{proof}[Proof of Theorem \ref{Thm:Bdd}(1), (2)]
Let $\calG$ denote the group $G$ and $m$ be a positive integer. Consider the action of $\calG$ on the set of vertices of $(C(G, S)^\sigma)^m$ by left multiplication, i.e., 
$$g\cdot v = gv$$
for any $g\in \calG, v\in G$. 

For an element $(s_1, \cdots, s_m)\in S^m$, let $\theta:G\to G$ denote the bijection defined by 
$$\theta(v) 
= 
\sigma^m(v) 
\prod_{i=1}^m 
\sigma^{m+1-i} (s_i)$$
for $v\in G$. 
For each $(s_1, \cdots, s_m)\in S^m$ and $v\in G$, note that 
\begin{align*}
\theta(g\cdot v) 
& =  
\sigma^m(g) \cdot \theta(v) \\
& = 
\psi_{(s_1, \cdots, s_m), v}(g) \cdot \theta(v) 
\end{align*}
for any $g\in \calG$, where $\psi_{(s_1, \cdots, s_m), v}$ denotes the automorphism 
$$g\mapsto 
\sigma^m(g)
$$
of the group $\calG$. Moreover, for any subset $A$ of $G$, 
$\scrN^m(g\cdot A) = \sigma^m(g) \cdot \scrN^m(A)$. 

Since $C(G, S)^\sigma$ is connected, its vertex Cheeger constant $h_\sigma$ is positive. Thus $C(G, S)^\sigma$ is an $h_\sigma$-expander with $h_\sigma>0$. 
If $\sigma^2$ is the trivial automorphism of $G$, then from Theorem \ref{thmPrincipal}, the nontrivial eigenvalues of the normalised adjacency operator of $C(G, S)^\sigma$ are greater than 
$$
-1 + \frac {h_\sigma^4}{2^{12} d^{8}}
.
$$
If $\sigma^{2k}$ is the trivial automorphism of $G$ for some odd integer $k\geq 1$, then from Theorem \ref{thmPrincipal}, the nontrivial eigenvalues of the normalised adjacency operator of $C(G, S)^\sigma$ are greater than 
$$
\left(
-1 + \frac {1}{2^{12} d^{8k}}
\left(
\frac 12 
\left(
1 - 
\left(
1 - \frac{h_\sigma^2}{2d^2}
\right)^k
\right)
\right)
^4
\right)^{1/k}.
$$
By the discrete Cheeger--Buser inequality (Proposition \ref{Prop:chin}), the result follows. 
\end{proof}

\subsection{The twisted Cayley sum graph}

For a subset $A$ of $G$, its neighbourhood in $C_\Sigma(G, S)^\sigma$ is denote by $\scrN_\Sigma(A)$. Given $g_1, \cdots, g_r \in G$,  $\prod_{i = 1}^r g_i$ denotes the product $g_1 \cdots g_r$. 

\begin{lemma}
\label{Lemma:UndirectedsigmaCayleySum}
The twisted Cayley sum graph $C_\Sigma (G, S)^\sigma$ is undirected if and only if $S$ contains $\sigma^2(g)\sigma(s)g^\mo$ for any $s\in S, g\in G$. 
\end{lemma}

\begin{proof}
If $h$ is adjacent to $g$, then $h = \sigma(g^\mo s)$ for some $s\in S$. Note that 
\begin{align*}
g 
& = s \sigma^\mo(h^\mo) \\
& = \sigma(h^\mo) \sigma(h) s\sigma^\mo(h^\mo),
\end{align*}
which implies that $g$ is adjacent to $h$ if and only if $\sigma(h) s\sigma^\mo(h^\mo) \in \sigma^\mo(S)$, i.e., 
$\sigma^2(h) \sigma(s) h^\mo \in S$. Hence $g$ is adjacent to each of its adjacent vertices if and only if $S$ contains 
$$\sigma^2(\sigma(g^\mo s)) \sigma(s) (\sigma(g^\mo s))^\mo
$$
for any $s\in S$. So $C_\Sigma (G, S)^\sigma$ is undirected if and only if $S$ contains $
\sigma^2(\sigma(g^\mo s)) \sigma(s) (\sigma(g^\mo s))^\mo
$ for any $s\in S, g\in G$. Note that for $s\in S$ and $x\in G$ and $g = (\sigma^\mo(x)s^\mo)^\mo$, 
$$
\sigma^2(\sigma(g^\mo s)) \sigma(s) (\sigma(g^\mo s))^\mo
=
\sigma^2(x) \sigma(s) x^\mo.
$$
So $C_\Sigma (G, S)^\sigma$ is undirected if and only if $S$ contains $\sigma^2(g)\sigma(s)g^\mo$ for any $s\in S, g\in G$. Hence the Lemma follows. 
\end{proof}

\begin{proof}[Proof of Theorem \ref{Thm:Bdd}(3)]
Let $\calG$ denote the group $G$ and $m$ be a positive integer. Consider the action of $\calG$ on the set of vertices of $(C_\Sigma(G, S)^\sigma)^m$ by right multiplication via the inverse, i.e., 
$$g\cdot v = vg^\mo$$
for any $g\in \calG, v\in G$. 

For an element $(s_1, \cdots, s_m)\in S^m$, let $\theta:G\to G$ denote the bijection defined by 
$$\theta(v) 
= 
\left(
\prod_{i = 1}^{\lfloor m/2\rfloor} \sigma^{2i} (s_{m+1 -2i}^\mo)
\right)
\sigma(v^{(-1)^m})
\left(
\prod_{i = 1}^{\lceil m/2\rceil} \sigma^{2\lceil m/2\rceil +1 - 2i} (s_{m+1 -(2\lceil m/2\rceil +1 - 2i)})
\right)
$$
for $v\in G$. 
For each $(s_1, \cdots, s_m)\in S^m$ and $v\in G$, set 
$$
\alpha = 
\left(
\prod_{i = 1}^{\lfloor m/2\rfloor} \sigma^{2i} (s_{m+1 -2i}^\mo)
\right),
\beta = 
\left(
\prod_{i = 1}^{\lceil m/2\rceil} \sigma^{2\lceil m/2\rceil +1 - 2i} (s_{m+1 -(2\lceil m/2\rceil +1 - 2i)})
\right),
$$
and note that 
\begin{align*}
\theta(g\cdot v) 
& =  
\alpha \sigma(gv^\mo) \beta \\
& =  
\alpha \sigma(v^\mo) \beta (\sigma(v^\mo) \beta)^\mo \sigma(gv^\mo) \beta \\
& =  
\alpha \sigma(v^\mo) \beta 
\left( 
(\sigma(v^\mo) \beta)^\mo \sigma(g^\mo) \sigma(v^\mo) \beta
\right)^\mo \\
& =  
\theta(v) 
\left( 
(\sigma(v^\mo) \beta)^\mo \sigma(g^\mo) \sigma(v^\mo) \beta
\right)^\mo \\
& =  
\left( 
(\sigma(v^\mo) \beta)^\mo \sigma(g^\mo) \sigma(v^\mo) \beta
\right)\cdot \theta(v) \\
& = 
\psi_{(s_1, \cdots, s_m), v}(g^\mo) \cdot \theta (v)
\end{align*}
for any $g\in \calG$, where $\psi_{(s_1, \cdots, s_m), v}$ denotes the automorphism 
$$g\mapsto 
(\sigma(v^\mo) \beta)^\mo \sigma(g) \sigma(v^\mo) \beta
$$
of the group $\calG$. 

For any subset $A$ of $G$, note that 
$$\scrN_\Sigma^{m} 
(A) 
= 
\sigma^2(S^\mo) \sigma^4(S^\mo) \cdots \sigma^{2\lfloor m/2\rfloor}(S^\mo) 
\sigma^{m}(A^{(-1)^m}) 
\sigma^{2\lceil m/2\rceil -1}(S) \cdots \sigma^3(S) \sigma(S).
$$
For any $g\in \calG$, one obtains 
\begin{align*}
& \sigma^{2m}(g) 
\sigma^{2m-1}(S) \cdots \sigma^3(S) \sigma(S)
g^\mo
\\
& = 
\left(
\sigma^{2m}(g) 
\sigma^{2m-1}(S)
\sigma^{2m-2}(g^\mo) 
\right)
\left(
\sigma^{2m-2}(g) 
\sigma^{2m-3}(S)
\sigma^{2m-4}(g^\mo) 
\right)
\cdots \\
& \qquad \qquad 
\left(
\sigma^6(g) 
\sigma^5(S)
\sigma^4(g^\mo) 
\right)
\left(
\sigma^4(g) 
\sigma^3(S)
\sigma^2(g^\mo) 
\right)
\left(
\sigma^2(g) 
\sigma(S)
g^\mo
\right)
\\
& = 
\sigma^{2m-2}(S) 
\cdots 
\sigma^4(S) 
\sigma^2(S)
S \\
& = 
\sigma^{2m-1}(S) 
\cdots 
\sigma^5(S) 
\sigma^3(S)
\sigma(S).
\end{align*}
So, for any subset $A$ of $G$, 
\begin{align*}
\scrN_\Sigma^m(\scrN_\Sigma^m(g\cdot A))
& = 
\sigma^2(S^\mo) \sigma^4(S^\mo) \cdots \sigma^{2m}(S^\mo) 
\sigma^{2m}(g\cdot A) 
\sigma^{2m-1}(S) \cdots \sigma^3(S) \sigma(S) \\
& = 
\sigma^2(S^\mo) \sigma^4(S^\mo) \cdots \sigma^{2m}(S^\mo) 
\sigma^{2m}(A) \sigma^{2m}(g^\mo)
\sigma^{2m-1}(S) \cdots \sigma^3(S) \sigma(S) \\
& = 
\sigma^2(S^\mo) \sigma^4(S^\mo) \cdots \sigma^{2m}(S^\mo) 
\sigma^{2m}(A) \sigma^{2m}(g^\mo)
\sigma^{2m-1}(S) \cdots \sigma^3(S) \sigma(S) gg^\mo\\
& = 
\sigma^2(S^\mo) \sigma^4(S^\mo) \cdots \sigma^{2m}(S^\mo) 
\sigma^{2m}(A) 
\sigma^{2m-1}(S) \cdots \sigma^3(S) \sigma(S) g^\mo\\
& = 
\scrN_\Sigma^m(\scrN_\Sigma^m(A)) g^\mo \\
& = 
g \cdot \scrN_\Sigma^m(\scrN_\Sigma^m(A)) 
\end{align*}
holds for any $g\in \calG$.

Since $C_\Sigma (G, S)^\sigma$ is connected, its vertex Cheeger constant $h_{\Sigma, \sigma}$ is positive. Thus $C_\Sigma (G, S)^\sigma$ is an $h_{\Sigma, \sigma}$-expander with $h_{\Sigma, \sigma} >0$. 
By Theorem \ref{thmPrincipal}, the nontrivial eigenvalues of the normalised adjacency operator of $C_\Sigma(G, S)^\sigma$ are greater than 
$$
-1 + \frac {h_{\Sigma, \sigma}^4}{2^{12} d^{8}}
.
$$
By the discrete Cheeger--Buser inequality (Proposition \ref{Prop:chin}), the result follows. 
\end{proof}

\section{Spectral expansion of the Cayley and Cayley sum graphs twisted by anti-automorphisms}
\label{Sec:TwistsAnti}

Let $G$ be a finite group, $S$ be a subset of $G$ and $\sigma$ be an group anti-automorphism of $G$.

Consider the twist $C(G, S)^\sigma$ of the Cayley graph $C(G, S)$ by the anti-automorphism $\sigma$. The graph $C(G, S)^\sigma$ has $G$ as its set of vertices, and there is an edge from $x$ to $y$ whenever $y = \sigma(xs)$ for some $s\in S$. Roughly speaking, the twisted Cayley graph $C(G, S)^\sigma$ has the same set of vertices as that of the Cayley graph $C(G, S)$, and given a vertex $x$ in $C(G, S)^\sigma$, its adjacent vertices are precisely the translates of the adjacent vertices of $x$ in $C(G, S)$ under $\sigma$. 

Consider the twist $C_\Sigma(G, S)^\sigma$ of the Cayley sum graph $C_\Sigma(G, S)$ by the anti-automorphism $\sigma$. The graph $C_\Sigma(G, S)^\sigma$ has $G$ as its set of vertices, and there is an edge from $x$ to $y$ whenever $y = \sigma(x^\mo s)$ for some $s\in S$. Roughly speaking, the twisted Cayley sum graph $C_\Sigma(G, S)^\sigma$ has the same set of vertices as that of the Cayley sum graph $C_\Sigma(G, S)$, and given a vertex $x$ in $C_\Sigma(G, S)^\sigma$, its adjacent vertices are precisely the translates of the adjacent vertices of $x$ in $C_\Sigma(G, S)$ under $\sigma$. 

\begin{theorem}\label{Thm:BddAnti}
Let $S$ be a subset of a finite group $G$ with $|S|= d$. Suppose $\sigma$ is an anti-automorphism of $G$. 
\begin{enumerate}
\item If the twisted Cayley graph $C (G, S)^\sigma$ is connected and undirected and $|G| \geq 4$, then the nontrivial spectrum of its normalised adjacency operator lies in the interval 
$$\left( -1 + \frac{h_\sigma^4}{2^{12}d^8}
,
1 - \frac{h_\sigma^2}{2d^2}
\right]$$
where $h_\sigma$ denotes the vertex Cheeger constant of $C (G, S)^\sigma$. 

\item 
Suppose $\sigma^{2}$ is the trivial automorphism of $G$. 
If the twisted Cayley sum graph $C_\Sigma(G, S)^\sigma$ is connected and undirected and $|G| \geq 4$, then the nontrivial spectrum of its normalised adjacency operator lies in the interval 
$$\left( -1 + \frac{h_{\Sigma, \sigma}^4}{2^{12}d^8} 
,
1 - \frac{h_{\Sigma, \sigma}^2}{2d^2}
\right]$$
where $h_{\Sigma, \sigma}$ denotes the vertex Cheeger constant of $C_\Sigma(G, S)^\sigma$. 

\item 
Suppose $\sigma^{2k}$ is the trivial automorphism of $G$, where $k\geq 1$ is an odd integer. 
If the twisted Cayley graph sum $C_\Sigma(G, S)^\sigma$ is connected and undirected and $|G| \geq 4$, then the nontrivial spectrum of its normalised adjacency operator lies in the interval 
$$\left( 
\left(
-1 + \frac {1}{2^{12} d^{8k}}
\left(
\frac 12 
\left(
1 - 
\left(
1 - \frac{h_{\Sigma, \sigma}^2}{2d^2}
\right)^k
\right)
\right)
^4
\right)^{1/k}
,
1 - \frac{h_{\Sigma, \sigma}^2}{2d^2}
\right]$$
where $h_{\Sigma, \sigma}$ denotes the vertex Cheeger constant of $C_\Sigma(G, S)^\sigma$. 
\end{enumerate}
\end{theorem}

\subsection{Cayley graph twisted by anti-automorphisms}

For a subset $A$ of $G$, its neighbourhood in $C(G, S)^\sigma$ is denote by $\scrN(A)$. Given $g_1, \cdots, g_r \in G$,  
$\prod_{i = 1}^r g_i$
denotes the product 
$g_1 \cdots g_r$. 

\begin{lemma}
\label{Lemma:UndirectedsigmaCayleySum}
The twisted Cayley graph $C (G, S)^\sigma$ is undirected if and only if $S$ contains $\sigma^2(g) \sigma(s^\mo) g^\mo$ for any $s\in S, g\in G$. 
\end{lemma}

\begin{proof}
If $h$ is adjacent to $g$, then $h = \sigma(g s)$ for some $s\in S$. Note that 
\begin{align*}
g 
& = \sigma^\mo(h)s^\mo \\
& = \sigma^\mo(h)s^\mo \sigma(h^\mo) \sigma(h) \\
& = \sigma(h^\mo \sigma^\mo(s^\mo) \sigma^{-2}(h))\sigma(h) \\
& = \sigma(h (h^\mo \sigma^\mo(s^\mo) \sigma^{-2}(h))),
\end{align*}
which implies that $g$ is adjacent to $h$ if and only if $h^\mo \sigma^\mo(s^\mo) \sigma^{-2}(h) \in S$. Hence $g$ is adjacent to each of its adjacent vertices if and only if $S$ contains 
$$
(\sigma(s) \sigma(g))^\mo \sigma^\mo(s^\mo) \sigma^{-2}(\sigma(s) \sigma(g))
$$
for any $s\in S$. So $C (G, S)^\sigma$ is undirected if and only if $S$ contains 
$$(\sigma(s) \sigma(g))^\mo \sigma^\mo(s^\mo) \sigma^{-2}(\sigma(s) \sigma(g))
= 
\sigma(g)^\mo \sigma(s^\mo) \sigma^\mo(s^\mo) \sigma^\mo(s) \sigma^\mo(g) 
= 
\sigma(g)^\mo \sigma(s^\mo) \sigma^\mo(g) 
$$
for any $s\in S, g\in G$. 
So $C (G, S)^\sigma$ is undirected if and only if $S$ contains $\sigma^2(g)\sigma(s^\mo)g^\mo$ for any $s\in S, g\in G$. Hence the Lemma follows. 
\end{proof}

\begin{proof}[Proof of Theorem \ref{Thm:BddAnti}(1)]
Let $\calG$ denote the group $G$. Consider the action of $\calG$ on the set of vertices of $C(G, S)^\sigma$ by left multiplication, i.e., 
$$g\cdot v = gv$$
for any $g\in \calG, v\in G$. 

For an element $s\in S$, let $\theta:G\to G$ denote the bijection defined by 
$$\theta(v) 
= 
\sigma(vs)$$
for $v\in G$. 
For each $s\in S$ and $v\in G$, note that 
\begin{align*}
\theta(g\cdot v) 
& =  
\sigma(gvs) \\
& =  
\sigma(gvs) (\sigma(vs))^\mo \sigma(vs) \\
& =  
(\sigma(vs) \sigma(g) (\sigma(vs))^\mo) \sigma(vs) \\
& = 
\psi_{s, v}(g) \cdot \theta(v) 
\end{align*}
for any $g\in \calG$, where $\psi_{s, v}$ denotes the anti-automorphism 
$$g\mapsto 
\sigma(vs) \sigma(g) (\sigma(vs))^\mo
$$
of the group $\calG$. 

For any subset $A$ of $G$, note that 
\begin{align*}
\scrN^{2}(A) 
& = \sigma(S) \sigma^2(A) \sigma^2(S),
\end{align*}
which yields 
\begin{align*}
\scrN^{2}(g\cdot A) 
& = 
\sigma(S) \sigma^2(g \cdot A) \sigma^2(S)\\
& = 
\sigma(S) \sigma^2(g)\sigma^2(A) \sigma^2(S)\\
& = 
g g^\mo \sigma(S) \sigma^2(g)\sigma^2(A) \sigma^2(S)\\
& = 
g (\sigma^2(g^\mo)\sigma(S^\mo)g)^\mo \sigma^2(A) \sigma^2(S)\\
& = 
g S^\mo \sigma^2(A) \sigma^2(S)\\
& = 
g \sigma(S) \sigma^2(A) \sigma^2(S)\\
& = 
g \cdot \scrN^2(A)  .
\end{align*}

Since $C(G, S)^\sigma$ is connected, its vertex Cheeger constant $h_\sigma$ is positive. Thus $C(G, S)^\sigma$ is an $h_\sigma$-expander with $h_\sigma>0$. 
By Theorem \ref{thmPrincipal}, the nontrivial eigenvalues of the normalised adjacency operator of $C(G, S)^\sigma$ are greater than 
$$
-1 + \frac {h_\sigma^4}{2^{12} d^{8}}
.
$$
By the discrete Cheeger--Buser inequality (Proposition \ref{Prop:chin}), the result follows. 
\end{proof}

\subsection{Cayley sum graph twisted by anti-automorphisms}

For a subset $A$ of $G$, its neighbourhood in $C_\Sigma(G, S)^\sigma$ is denote by $\scrN_\Sigma(A)$. Given $g_1, \cdots, g_r \in G$,  
$\prod_{i = 1}^r g_i$
denotes the product 
$g_1 \cdots g_r$. 

\begin{proof}[Proof of Theorem \ref{Thm:BddAnti}(2), (3)]
Let $\calG$ denote the group $G$. Consider the action of $\calG$ on the set of vertices of $C_\Sigma(G, S)^\sigma$ by right multiplication via the inverse, i.e., 
$$g\cdot v = vg^\mo$$
for any $g\in \calG, v\in G$. 

For an element $(s_1, \cdots, s_m)\in S^m$, let $\theta:G\to G$ denote the bijection defined by 
$$\theta(v) 
= 
\prod_{i=1}^m 
\sigma^i (s^{(-1)^{i-1}}_{m+1-i})
\sigma^m(v^{(-1)^m}) $$
for $v\in G$. 
For each $(s_1, \cdots, s_m)\in S^m$ and $v\in G$, note that 
\begin{align*}
\theta(g\cdot v) 
& =  
\prod_{i=1}^m 
\sigma^i (s^{(-1)^{i-1}}_{m+1-i})
\sigma^m((g\cdot v) ^{(-1)^m})\\
& =  
\prod_{i=1}^m 
\sigma^i (s^{(-1)^{i-1}}_{m+1-i})
\sigma^m(v^{(-1)^m})
(\sigma^m(v^{(-1)^m}))^\mo
\sigma^m((g\cdot v) ^{(-1)^m})\\
& =  
\theta(v) 
(\sigma^m(v^{(-1)^m}))^\mo
\sigma^m((g\cdot v) ^{(-1)^m})\\
& =  
\theta(v) 
\sigma^m(g^{(-1)^{m-1}}) \\
& =  
\sigma^m(g^{(-1)^m})
\cdot \theta(v) \\
& = 
\psi_{(s_1, \cdots, s_m)}(g) \cdot \theta(v) 
\end{align*}
for any $g\in \calG$, where $\psi_{(s_1, \cdots, s_m), v}$ denotes the automorphism 
$$g\mapsto 
\sigma^m(g^{(-1)^m})
$$
of the group $\calG$. 

For any integer $m\geq 1$ and any subset $A$ of $G$, note that 
\begin{align*}
\scrN_{\Sigma}^{2m}(A) 
& = \sigma(S) \sigma^2(S^\mo) \sigma^3(S) \sigma^4(S^\mo) \cdots \sigma^{2m-1} (S) \sigma^{2m} (S^\mo)\sigma^{2m}(A) ,
\end{align*}
which yields 
\begin{align*}
\scrN_{\Sigma}^{2m}(g\cdot A) 
& = \sigma(S) \sigma^2(S^\mo) \sigma^3(S) \sigma^4(S^\mo) \cdots \sigma^{2m-1} (S) \sigma^{2m} (S^\mo)\sigma^{2m}(g\cdot A) \\
& = \sigma(S) \sigma^2(S^\mo) \sigma^3(S) \sigma^4(S^\mo) \cdots \sigma^{2m-1} (S) \sigma^{2m} (S^\mo)\sigma^{2m}(Ag^\mo) \\
& = \sigma^{2m}(g) \cdot (\sigma(S) \sigma^2(S^\mo) \sigma^3(S) \sigma^4(S^\mo) \cdots \sigma^{2m-1} (S) \sigma^{2m} (S^\mo)\sigma^{2m}(A)) \\
& = \sigma^{2m}(g) \cdot \scrN_{\Sigma}^{2m}(A) .
\end{align*}

Since $C_\Sigma(G, S)^\sigma$ is connected, its vertex Cheeger constant $h_{\Sigma, \sigma}$ is positive. Thus $C_\Sigma(G, S)^\sigma$ is an $h_{\Sigma, \sigma}$-expander with $h_{\Sigma, \sigma}>0$. 
If $\sigma^2$ is the trivial automorphism of $G$, then from Theorem \ref{thmPrincipal}, the nontrivial eigenvalues of the normalised adjacency operator of $C_\Sigma(G, S)^\sigma$ are greater than 
$$
-1 + \frac {h_{\Sigma, \sigma}^4}{2^{12} d^{8}}
.
$$
If $\sigma^{2k}$ is the trivial automorphism of $G$ for some odd integer $k\geq 1$, then from Theorem \ref{thmPrincipal}, the nontrivial eigenvalues of the normalised adjacency operator of $C_\Sigma(G, S)^\sigma$ are greater than 
$$
\left(
-1 + \frac {1}{2^{12} d^{8k}}
\left(
\frac 12 
\left(
1 - 
\left(
1 - \frac{h_{\Sigma, \sigma}^2}{2d^2}
\right)^k
\right)
\right)
^4
\right)^{1/k}.
$$
By the discrete Cheeger--Buser inequality (Proposition \ref{Prop:chin}), the result follows. 
\end{proof}

\section{Spectral expansion of Schreier graphs}
\label{Sec:Schreier}

Given a subgroup $H$ of a finite group $G$, and a symmetric subset $S$ of $G$ with $|S| = d$, the Schreier graph $\mathrm{Sch} (G, H, S)$ has the set $H\backslash G$ of right cosets of $H$ in $G$ as its set of vertices and there is an edge from $Hg$ to $Hg'$ for $g, g'\in G$, if $Hg' = Hgs$ for some $s\in S$. 

\begin{theorem}
\label{Thm:Schreier}
Suppose no index two subgroup of $G$ acts transitively on the vertex set of $\mathrm{Sch} (G, H, S)$. Assume that the index of $H$ in $G$ is at least $4$, and $S\cdot S$ contains its conjugates by the elements of $G$. If the Schreier graph $\mathrm{Sch} (G, H, S)$ is connected, then the nontrivial spectrum of its normalised adjacency operator lies in the interval 
$$\left( -1 + \frac{h_{\mathrm{Sch}}^4}{2^{12}d^8}
,
1 - \frac{h_{\mathrm{Sch}}^2}{2d^2}
\right]$$
where $h_{\mathrm{Sch}}$ denotes the vertex Cheeger constant of $\mathrm{Sch} (G, H, S)$. 
\end{theorem}

\begin{proof}
Let $\calG$ denote the group $G$. Consider the action of $\calG$ on the set of vertices $H\backslash G$ of $\mathrm{Sch} (G, H, S)$ by right multiplication via the inverse, i.e., 
$$\tau \cdot Hg
= Hg \tau^\mo
$$
for any $\tau \in \calG$ and any right coset $Hg$ of $H$ in $G$. 
Note that this action is transitive of order $|H|$. 

For an element $s\in S$, let $\theta: 
H\backslash G 
\to 
H\backslash G$
denote the bijection defined by 
$$\theta(Hg) 
= Hgs.$$
For any $Hg\in H\backslash G$, note that 
\begin{align*}
\theta(\tau \cdot Hg)
& = Hg\tau^\mo s\\
& = Hgs (s^\mo \tau s)^\mo \\
& = (s^\mo \tau s) \cdot Hgs\\
& = \psi_s(\tau) \cdot Hgs\\
\end{align*}
for any $\tau \in \calG$, where $\psi_s$ denote the automorphism 
$$\tau \mapsto s^\mo \tau s$$
of the group $\calG$. 

For any element $Hg\in H\backslash G$, note that 
\begin{align*}
\scrN_S(\scrN_S(\tau\cdot Hg))
& = \{Hg\tau^\mo x\,|\, x\in S\cdot S\} \\
& = \{Hg\tau^\mo x\,|\, x\in \tau S\cdot S \tau^\mo \} \\
& = \{Hgx\tau^\mo \,|\, x\in S\cdot S\} \\
& = \{\tau \cdot Hgx \,|\, x\in S\cdot S\} \\
& = \tau \cdot \scrN_S(\scrN_S(Hg))
\end{align*}
holds for any $\tau \in \calG$. 

If the Schreier graph $\mathrm{Sch} (G, H, S)$ is an $\varepsilon$-vertex expander for some $\varepsilon>0$, then by Theorem \ref{thmPrincipalk1}, the nontrivial eigenvalues of its normalised adjacency operator are greater than 
$$ - 1 + \frac{\varepsilon^4} {2^{12} d^8}.$$
By the discrete Cheeger--Buser inequality (Proposition \ref{Prop:chin}), the result follows. 
\end{proof}

\def\cprime{$'$} \def\Dbar{\leavevmode\lower.6ex\hbox to 0pt{\hskip-.23ex
  \accent"16\hss}D} \def\cfac#1{\ifmmode\setbox7\hbox{$\accent"5E#1$}\else
  \setbox7\hbox{\accent"5E#1}\penalty 10000\relax\fi\raise 1\ht7
  \hbox{\lower1.15ex\hbox to 1\wd7{\hss\accent"13\hss}}\penalty 10000
  \hskip-1\wd7\penalty 10000\box7}
  \def\cftil#1{\ifmmode\setbox7\hbox{$\accent"5E#1$}\else
  \setbox7\hbox{\accent"5E#1}\penalty 10000\relax\fi\raise 1\ht7
  \hbox{\lower1.15ex\hbox to 1\wd7{\hss\accent"7E\hss}}\penalty 10000
  \hskip-1\wd7\penalty 10000\box7}
  \def\polhk#1{\setbox0=\hbox{#1}{\ooalign{\hidewidth
  \lower1.5ex\hbox{`}\hidewidth\crcr\unhbox0}}}
\providecommand{\bysame}{\leavevmode\hbox to3em{\hrulefill}\thinspace}
\providecommand{\MR}{\relax\ifhmode\unskip\space\fi MR }
% \MRhref is called by the amsart/book/proc definition of \MR.
\providecommand{\MRhref}[2]{%
  \href{http://www.ams.org/mathscinet-getitem?mr=#1}{#2}
}
\providecommand{\href}[2]{#2}

% \bibliography{../../../2020BibShort/biblio}

% \bibliographystyle{amsalpha}

\end{document}